\documentclass{amsart} 
\usepackage{amsthm, amssymb, amsmath}       
\usepackage[numbers]{natbib}
\usepackage{marginnote,caption,subcaption}
\usepackage{tikz,verbatim,xcolor,graphicx}

\theoremstyle{plain}       
\newtheorem{theorem}{Theorem}[section]
\newtheorem{lemma}[theorem]{Lemma}
\newtheorem{definition}[theorem]{Definition}
\newtheorem{proposition}[theorem]{Proposition}
\newtheorem{conjecture}[theorem]{Conjecture}

\renewenvironment{equation} 
  {\begin{equation*}}
  {\end{equation*}}
\makeatother

\DeclareMathOperator{\prank}{prank}
\DeclareMathOperator{\arank}{arank}
\DeclareMathOperator{\bias}{bias}

\DeclareMathOperator{\im}{im}
\DeclareMathOperator{\Span}{Span}
\newcommand\F{\mathbb{F}}

\begin{document}
\vspace{10mm}

\title[Restricted Random Set Progressions]
    {A note on lower bounds in Szemer\'edi's theorem with random differences}
\author{Jason Zheng}
\address{University of Michigan}
\email{jkzheng@umich.edu}
\date{Aug 2, 2025}

\begin{abstract}
In this note, we consider Szemer\'{e}di's theorem on $k$-term arithmetic progressions over finite fields $\mathbb{F}_p^n$, where the allowed set $S$ of common differences in these progressions is chosen randomly of fixed size. Combining a generalization of an argument of Altman with Moshkovitz--Zhu's bounds for the partition rank of a tensor in terms of its analytic rank, we (slightly) improve the best known lower bounds (due to Bri\"et) on the size $|S|$ required for Szemer\'edi's theorem with difference in $S$ to hold asymptotically almost surely. 
\end{abstract}

\maketitle

\section{Introduction}
 In 1975, Szemer\'{e}di \cite{szemThm} famously proved that dense subsets of the integers contain $k$-term arithmetic progressions. Where Szemer\'edi's theorem allows any common difference $d$ in the arithmetic progression $(x,x+d, \dots, x+(k-1)d)$, there is substantial interest in determining sets $S$ for which Szemer\'edi's theorem holds with the additional restriction that $d\in S$. Szemer\'edi's theorem says that setting $S= \mathbb{N}$ suffices. In this paper we are interested in understanding how sparse $S$ can be for such a statement to hold. In particular, we are interested in determining the minimal density of a random $S$ such that $S$ satisfies this property with high probability. Frantzikinakis, Lesigne, and Weirdl \cite{frantzConj} provide the following interesting conjecture. The asymptotic notation $\omega$ that is used here satisfies that $f=o(g)$ if and only if $g=\omega(f)$. 
 \begin{conjecture}
     Let $S \subset \mathbb{N}$ be chosen at random with $\mathbb{P}[d \in S] = \omega(\frac{1}{d})$. Then asymptotically almost surely, all subsets of $\mathbb{N}$ with positive upper density contain a $k$-term arithmetic progression with common difference in $S$.
 \end{conjecture}

 We direct the reader to \cite[Chapter~11]{taoVu} for more on this problem in the integers, but remark that in this setting, asymptotics for $|S\cap \{1,\ldots, N\}|$ are known only for $k=2$. One may ask a similar question in any finite abelian group. A setting of particular interest in additive combinatorics and related fields is the finite field model setting $\mathbb{F}_p^n$ ($p$ fixed, $n$ large) \cite{benGreenExposition}.

  In this setting, we form $S\subset \mathbb{F}_p^n$ by including elements with equal probability. Our argument is inspired by \cite{dan}, who showed a lower bound of $\binom{n+1}{2} - Cn\log_pn$ in the case of $k = 3$ by considering certain vector spaces of matrices and bounds on matrix rank. In 2021, \cite{briet} generalized the argument of \cite{dan} and together with a new ingredient on subspaces of tensors possessing high analytic rank, showed a lower bound of $\binom{n + k - 2}{k - 1} - C(\log_pn)^2n^{k-2}$. We show in this paper that one may instead more directly generalize the argument of \cite{dan} and use bounds between the analytic and partition rank of tensors. In doing so, we obtain that random sets $S$ of size $\binom{n + k - 2}{k - 1} - C(\log_pn)^{1 + \epsilon}n^{k-2}$  yield (with high probability) dense subsets of $\mathbb{F}_p^n$ with no $k$-APs with common difference in $S$, thus slightly improving on the lower bound from \cite{briet}. We note that this improvement relies upon improved bounds \cite{MZ} relating the analytic rank of a tensor to its partition rank, which were not available at the time of writing of \cite{briet}.

We now formally state our main result.
\begin{theorem} \label{thm:Big Theorem}
    For every integer $k \geq 3$, prime $p \geq k$, there is a constant $C_{p,k}>0$ and a function $\epsilon: \mathbb{R}^+ \to \mathbb{R}^+$ with  $\epsilon(x)\to 0$ as $x\to \infty$ such that the following holds. If $S \subset \mathbb{F}_p^n$ is a set formed by selecting at most 
    \begin{equation}
        \binom{n+k-2}{k-1} - C_{p,k}(\log_pn)^{1 + \epsilon(n)}n^{k - 2}
    \end{equation}elements independently and uniformly at random, then with probability $1 - o_{n\to\infty}(1)$ there is a set $A \subset \mathbb{F}_p^n$ of size $|A| \geq \Omega_{k,p}(p^n)$ that contains no proper $k$-term arithmetic progression with common difference in $S$.
\end{theorem}

\subsection*{Acknowledgement}
    Part of this work was conducted as part of the 2024 University of Michigan REU, and the author is grateful for its support. The author also thanks Dan Altman for his guidance throughout the project, and Sarah Peluse and Nathan Tung for feedback on an earlier version of this paper.

\section{On ranks of tensors}
    Before we prove Theorem \ref{thm:Big Theorem}, this section establishes our preliminary facts and definitions regarding tensors. In the following we use the terminology $d$-tensor and $d$-linear form interchangeably. We also pass liberally between interpreting a $d$-tensor as a multilinear form and as the corresponding $d$-dimension box of coefficients defining it. 
    
  \begin{definition} [Partition rank]
    Let $d \geq 2$, $n \geq 1$ be integers. A $d$-linear form $T: \mathbb{F}^n \times \dots \times \mathbb{F}^n \rightarrow \mathbb{F}$ has partition rank $1$ if there exist integers $1 \leq a, b \leq d - 1$ such that $a+b = d$, a partition $\{ i_1, \dots, i_a \}, \{ j_1, \dots, j_b \}$ of $[d]$, and $a,b$-linear forms $T_1, T_2$ (respectively) such that for any $x_1, \dots, x_d \in \mathbb{F}^n$, 
    \begin{equation}
        T(x_1, \dots, x_d) = T_1(x_{i_1}, \dots, x_{i_a})T_2(x_{j_1}, \dots, x_{j_b}).
    \end{equation}
    The partition rank of $T$, denoted $\prank(T)$, is the smallest $r$ such that $T$ can be expressed as $T = T_1 + \dots + T_r$, where each $T_i$ has partition rank $1$.
\end{definition}

\begin{lemma} \label{lem:low prank bound}
    For integers $n\geq d \geq 3$, the number of $d$-tensors on $\mathbb{F}_p^n$ of partition rank at most $r$ is at most $p^{2n^{d-1}r}$.
\end{lemma}
\begin{proof}
    For a natural number $d$, let $f(d)$ denote the number of $d$-tensors on $\mathbb{F}_p^n$ and let $g(d)$ denote the number of $d$-tensors of partition rank exactly $1$.
    Counting choices for the value of $a$ and the input set of $T_1$, we obtain the bound 
    \begin{equation}
        g(d) \leq \sum_{a = 1}^{d-1} f(a) f(d - a) \binom{d}{a}.
    \end{equation}Note that $f(x) = p^{n^x}$, and that $n^a + n^{d - a} \leq n + n^{d-1}$, to obtain the bound
    \begin{equation}
        g(d) \leq 2^d \sum_{a = 1}^{d-1} p^{n^a + n^{d - a}} \leq dp^dp^{n+n^{d-1}} \leq p^{2n^{d-1}}.
    \end{equation}Now, we recall that any $d$-tensor $T$ of partition rank at most $r$ can be expressed as $T = T_1 + \dots + T_r$ for partition rank-$1$ tensors $T_i$. Thus, the total number of these tensors is at most $g(d)^r \leq p^{2n^{d-1}r}$.
\end{proof}
We also use a more analytic notion of tensor rank, introduced by Gowers and Wolf in \cite{GW11}.
\begin{definition} [Bias and analytic rank]
    Let $d \geq 2$, $n \geq 1$ be integers. Let $\mathbb{F}$ be a finite field and let $\mathbb{\chi}: \mathbb{F} \rightarrow \mathbb{C}$ be a nontrivial additive character. Let $T \in \mathbb{F}^{n \times \dots \times n}$ be a $d$-tensor. Then, the bias of $T$ is defined by 
    \begin{equation}
        \bias(T) = \mathbb{E}_{x_1, \dots, x_d \in \mathbb{F}^n}\chi(T(x_1, \dots, x_d)),
    \end{equation}
    and the analytic rank of $T$, denoted $\arank(T)$, is defined by 
    \begin{equation}
        \arank(T) = -\log_{|\mathbb{F}|}\bias(T).
    \end{equation}
\end{definition}

It is true but not trivial that these notions of rank are quite closely related.

 \begin{proposition} [\cite{kazhdan2018appoximatecohomology}, \cite{lovett2019}]  \label{prop: arank leq prank} 
     For any $d$-tensor $T$, $\arank{T} \leq \prank{T}$.
 \end{proposition}

 It is an open problem to obtain linear bounds in the other direction, and the exponent of $1 + \epsilon$ in the log factor of the expression from Theorem $\ref{thm:Big Theorem}$ follows directly from the corresponding bound between partition and analytic rank. If the best bound between partition and analytic rank is improved, our result also improves without additional input. Furthermore, should the linear relationship between analytic and partition rank be proven, we would get bounds for $k \geq 4$ with linear dependence on the logarithm in the lower order term, similar to that of \cite{dan} for $k = 3$. Nonetheless, the following is state-of-the-art at the time of writing of this paper.

 \begin{theorem} [Relationship between partition and analytic rank \cite{MZ}] \label{thm: prank arank bound}
     For any $d \geq 2$, there exists $\alpha_d \geq 1$ and a function $\epsilon_d:\mathbb{R}^+\to \mathbb{R}^+$ with $\lim_{x\to \infty} \epsilon_d(x) = 0$ such that for any nonzero $d$-tensor $T$,
     \begin{equation}
         \prank{T} \leq \alpha_d \cdot \arank{T}(\log(1 + \arank{T})+1) \leq \alpha_d(\arank{T})^{1 + \epsilon(\arank{T})}.
     \end{equation}
    
 \end{theorem}

Using the notation of Theorem \ref{thm: prank arank bound}, it suffices to assume $\epsilon_d$ is nonincreasing, and for the rest of this paper, we do. When $d$ is fixed, let $\alpha = \alpha_d$, $\epsilon = \epsilon_d \leq o(1)$ be as such.

\section{Proof of Theorem \ref{thm:Big Theorem}}

We say that a tensor $T$ is symmetric if it is invariant under permutation of its inputs. Note that the space of symmetric $d$-tensors has dimension $\binom{n+d-1}{d}$. Let $\phi_d$ denote the degree $d$ Veronese map $\F_p^n \to \F_p^{\binom{n+d-1}{d}}$. Then we may view a symmetric tensor $T$ as acting linearly on the image of $\phi_d$, and in particular introduce the inner product $\langle \cdot , \cdot \rangle$ and vector 
\[v_T\in \left( \F_p^{\binom{n+d-1}{d}}\right)^{\ast}\]
by $T(x,\ldots, x) = \langle v_T, \phi_d(x) \rangle$. 

In what follows we will be interested in $d=k-1$ tensors, corresponding to $k$APs. We state the upcoming lemmas in terms of the variable $k$ to highlight the dependence on the length of the arithmetic progression, and then afterwards pass to the variable $d$ for brevity in computation. 

We begin by recording the following lemma of \cite{dan}, which was also used in \cite{briet}. The proof is linear algebra and we omit the details.

\begin{lemma}[{\cite[Lemma 3.3]{dan}}]
Let $k \geq 3$ be an integer, $p \geq k$ prime. Let $S \subset \mathbb{F}_p^n$ be such that the set $\phi_{k - 1}(S)$ is linearly independent. Then there exists a nonzero symmetric $(k-1)$-tensor $T$ such that the set $\{ x \in \mathbb{F}_p^n: T(x,\ldots,x) = 0 \}$ contains no $k$-term arithmetic progressions with common difference in $S$. 
\end{lemma}

Furthermore, it is a standard fact (which follows from the Chevalley--Warning theorem) that the set of $x$ such that $T(x,\ldots, x)=0$ has size $\Omega_{p,k}(p^n)$. Therefore, to prove Theorem \ref{thm:Big Theorem}, it suffices to show that with high probability, $S$ is such that $\phi_{k-1}(S)$ is linearly independent. To this end, the following suffices.

\begin{lemma} \label{lem:big lemma}
For every integer $k \geq 3$, prime $p \geq k$, there is $\gamma_k \geq 0$, $\epsilon: \mathbb{R}^+ \to \mathbb{R}^+$ such that $\epsilon \leq o(1)$ and the following holds. Let \[s \leq \binom{n + k - 2}{k - 1} - \gamma_k(\log_pn)^{1 + \epsilon(n)}n^{k - 2}\] be a positive integer. Let $x_1, \dots, x_s$ be independent and uniformly distributed random vectors in $\mathbb{F}_p^n$. Then $\phi_{k-1}(x_1), \dots, \phi_{k-1}(x_s)$ are linearly independent with probability $1 - o_{n\to \infty}(1)$.
\end{lemma}
    
The rest of this paper proves Lemma \ref{lem:big lemma}. For the rest of this section, for brevity we set $d := k - 1$, since we only deal with $(k-1)$-tensors in the proof of our main result for $k$-arithmetic progressions. Furthermore, we allow implicit constants to depend on $p$ and $d$. Select a positive integer 
\[
s \le \binom{n+d-1}{d}
  - \beta_{d}(\log_p n)^{1+\epsilon(n)}n^{d-1},
\]
where $\beta_d$ is some constant which we will choose later, and $\epsilon=\epsilon_d$ is the function from Theorem \ref{thm: prank arank bound}. 
Let $x_1, \dots, x_s \in \mathbb{F}_p^n$ be independently, uniformly distributed vectors. Let $U$ be a subspace of $\F_p^{\binom{n+d-1}{d}}$ with dimension $s-1$ and which maximally intersects $\im(\phi_{d})$ (among all subspaces of dimension $s - 1$). We record the following lemma which features as \cite[Lemma 3.4]{dan}.

    \begin{lemma}
        The probability that $\phi_d(x_1), \dots, \phi_d(x_s)$ are linearly independent is bounded below by $(1 - \mathbb{P}_{x \in \mathbb{F}_p^n}[\phi_d(x) \in U] )^s.$
    \end{lemma}
    \begin{proof}
    The probability is bounded below by
        \begin{equation}
            \mathbb{P}[x_1 \neq 0]\prod_{i = 2}^s \mathbb{P}[\phi_d(x_i) \notin \Span{} \{ \phi_d(x_1), \dots, \phi_d(x_{i-1}) \} ]
        \end{equation}
        \begin{equation}
            \geq \mathbb{P}[x_1 \neq 0] \prod_{i = 2}^s \mathbb{P}[\phi_d(x_i) \notin U]
        \end{equation}
        \begin{equation}
            \geq (1 - \mathbb{P}_{x \in \mathbb{F}_p^n}[\phi_d(x) \in U] )^s.
        \end{equation}
    \end{proof}

    We claim therefore that to show that $\phi_d(x_1), \dots, \phi_d(x_s)$ are linearly independent with high probability, it suffices to show that $\mathbb{P}_{x \in \mathbb{F}_p^n}[\phi_d(x) \in U] \leq o(n^{-d})$. Indeed,  $s \leq \binom{n + d - 1}{d} \leq O(n^d)$, and thus by the previous lemma, the probability of linear independence is bounded below by $(1 - o(n^{-d}))^{O(n^d)} = 1 - o(1)$. 
    
    We now show that $\mathbb{P}_{x \in \mathbb{F}_p^n}[\phi_d(x) \in U] \leq o(n^{-d})$. First, we will need \cite[Lemma 3.5]{dan}. This follows from the orthogonality of characters.
    \begin{lemma} \label{lem: Dan 3.5}
        Let $V$ be a subspace of the vector space of functions $\mathbb{F}_p^k \to \mathbb{F}_p$. Let $\chi$ be a nontrivial character on $\F_p$. Say $V(x) = 0$ if $v(x) = 0$ for all $v \in V$. Then
        \begin{equation}
            \mathbb{P}_x (V(x) = 0) = \mathbb{E}_{v,x} \chi(v(x)).
        \end{equation}
    \end{lemma}

    We also record the following lemma of Gowers and Wolf which bounds the bias of $T(x,x,\ldots, x)$ on $\F_p^n$ by the analytic rank of the tensor $T$.

    \begin{lemma}[{\cite[Lemma 3.2]{GW11}}]
        Let $\chi$ be a nontrivial additive character on $\F_p$ and $T$ a $d$-tensor on $\F_p^n$. Then
        \[\left| \mathbb{E}_{x\in \F_p^n} \chi(T(x,\ldots,x))\right| \leq p^{-\arank{T}/2^{d-1}}.\]
    \end{lemma}
    
    Combining the above two lemmas we obtain    
    \begin{align*}
        \mathbb{P}_{x \in \mathbb{F}_p^n}[\phi_d(x) \in U] &= \mathbb{P}_x [\langle v_T , \phi_d(x) \rangle = 0, \text{ for all } v_T \in U^\perp] \\
        &= \mathbb{E}_{x, v_T \in U^\perp} [\chi(\langle v_T, \phi_d(x) \rangle)]  \\ &\leq \mathbb{E}_{T \in U^\perp} [p^{-\frac{\arank{T}}{2^{d-1}}}].
    \end{align*}Thus, it suffices to show that 
   \begin{equation}
       \mathbb{E}_{T\in U^\perp} [p^{-\frac{\arank{T}}{2^{d-1}}}] = o(n^{-d}).
   \end{equation}We do so by splitting the sum by analytic rank and making use of Lemma \ref{lem:low prank bound} to show that there are few tensors with small analytic rank. Let $r_0$ be a parameter which we will choose shortly and let $r$ be such that Theorem \ref{thm: prank arank bound} yields that $\arank{T} \leq r_0$ implies $\prank{T} \leq r$ (so we may take $r=\alpha_d r_0^{1+\epsilon(n)}$).  

    \begin{align*}
    \sum_{T \in U^\perp}p^{-\frac{\arank T}{2^{d-1}}}
    &= \sum_{\arank T \le r_0 \text{ or } \arank T > r_0}
       p^{-\frac{\arank T}{2^{d-1}}} \\[6pt]
    &\le \sum_{\prank T \le r}
       p^{-\frac{\arank T}{2^{d-1}}}
       + \sum_{\arank T > r_0}
       p^{-\frac{\arank T}{2^{d-1}}} \\[6pt]
    &\le \bigl|\{T:\prank T \le r\}\bigr|
       + p^{-\frac{r_0}{2^{d-1}}}\,
         \bigl|\{T\in U^\perp:\arank T > r_0\}\bigr| \\[6pt]
    &\le p^{2n^{d-1}r}
       + p^{\dim U^\perp}\ \cdot p^{-\frac{r_0}{2^{d-1}}},\
    \end{align*}where we use Lemma \ref{lem:low prank bound} in the final line. Dividing by $p^{\dim U^\perp}$ we obtain that  
    \begin{equation}
        \mathbb{E}_{T\in U^\perp} [p^{-\frac{\arank{T}}{2^{d-1}}}]
        \leq 
        p^{2n^{d-1}r - \dim U^\perp} + p^{-\frac{r_0}{2^{d-1}}}.
    \end{equation}We consider the two terms separately. Setting 
    \[ r_0 = (d2^{d-1}+1)\log_p n,\]
    we ensure that the second term is of size $o(n^{-d})$. For the first term, we compute that the exponent of $p$ is 
\begin{equation}
        2n^{d-1}r - \dim U^\perp \leq
        2n^{d-1}\alpha_d r_0^{1+\epsilon(n)} - \beta_d (\log_p n)^{1 + \epsilon(n)}
        n^{d-1}.
\end{equation}Therefore, setting $\beta_d$ suitably large we may ensure that this exponent is bounded above by $-n^{d-1}(\log_p n)^{1+\epsilon(n)}$ (say), which certainly ensures that the first term is bounded above by $o(n^{-d})$.

Thus, we conclude that $\mathbb{E}_{T\in U^\perp}[p^{-\frac{\arank T}{2^{d-1}}}] \leq o(n^{-d})$, as desired. This completes the proof of Theorem \ref{thm:Big Theorem}.

\begingroup
\large
\bibliographystyle{amsalpha}
\AtBeginEnvironment{bibliography.bib}{\large}
\bibliography{bibliography.bib}
\endgroup

\end{document}